\documentclass[11pt, reqno]{amsart}

\author[P.~Leonetti]{Paolo Leonetti}
\address{Universit\`a ``Luigi Bocconi''\\Department of Statistics\\Milan, Italy}
\email{leonetti.paolo@gmail.com}
\urladdr{\url{http://orcid.org/0000-0001-7819-5301}}

\author[C.~Sanna]{Carlo Sanna}
\address{Universit\`a degli Studi di Torino\\Department of Mathematics\\Turin, Italy}
\email{carlo.sanna.dev@gmail.com}
\urladdr{\url{http://orcid.org/0000-0002-2111-7596}}


\title{On the $p$-adic valuation of Stirling~numbers\\ of the first kind}

\usepackage{amsmath}
\usepackage{amssymb}
\usepackage{amsthm}
\usepackage[left=3.5cm, right=3.5cm, paperheight=11.8in]{geometry}
\usepackage{hyperref}
\usepackage{fancyhdr}
\usepackage{enumitem}
\usepackage{graphicx}

\newtheorem{thm}{Theorem}[section]
\newtheorem{lem}[thm]{Lemma}
\newtheorem{cor}[thm]{Corollary}
\newtheorem{conj}{Conjecture}[section]
\theoremstyle{definition}

\hypersetup{
    pdftitle={On the p-adic valuation of Stirling numbers of first kind},
    pdfauthor={Paolo Leonetti and Carlo Sanna},
    pdfmenubar=false,
    pdffitwindow=true,
    pdfstartview=FitH,
    colorlinks=true,
    linkcolor=blue,
    citecolor=green,
    urlcolor=cyan
}

\def\free{\operatorname{free}}
                                 
\begin{document}

\maketitle
\thispagestyle{empty}

\begin{abstract}
For all integers $n \geq k \geq 1$, define $H(n,k) := \sum 1 / (i_1 \cdots i_k)$, where the sum is extended over all positive integers $i_1 < \cdots < i_k \leq n$. 
These quantities are closely related to the Stirling numbers of the first kind by the identity $H(n,k) = s(n + 1, k + 1) / n!$.
Motivated by the works of Erd\H{o}s--Niven and Chen--Tang, we study the $p$-adic valuation of $H(n,k)$.
In particular, for any prime number $p$, integer $k \geq 2$, and $x \geq (k-1)p$, we prove that $\nu_p(H(n,k)) < -(k - 1)(\log_p(n/(k - 1)) - 1)$ for all positive integers $n \in [(k-1)p, x]$ whose base $p$ representations start with the base $p$ representation of $k - 1$, but at most $3x^{0.835}$ exceptions. 
We also generalize a result of Lengyel by giving a description of $\nu_2(H(n,2))$ in terms of an infinite binary sequence.
\end{abstract}

\section{Introduction}

It is well known that the $n$-th harmonic number $H_n := 1+\frac{1}{2}+\cdots+\frac{1}{n}$ is not an integer whenever $n \geq 2$.
Indeed, this result has been generalized in several ways (see, e.g.,~\cite{MR2385421, zbMATH03009590, zbMATH02614460}).
In particular, given integers $n \geq k \geq 1$, Erd\H{o}s and Niven~\cite{MR0015401} proved that
\begin{equation*}\label{eq:Hnk}
H(n, k) := \sum_{1 \leq i_1 < \cdots < i_k \leq n} \frac1{i_1 \cdots i_k}
\end{equation*}
is an integer only for finitely many $n$ and $k$.
Precisely, Chen and Tang~\cite{MR2999589} showed that $H(1,1)$ and $H(3,2)$ are the only integral values. (See also~\cite{MR3267180} for a generalization to arithmetic progressions.)

A crucial step in both the proofs of Erd\H{o}s--Niven and Chen--Tang's results consists in showing that, when $n$ and $k$ are in an appropriate range, for some prime number $p$ the $p$-adic valuation of $H(n,k)$ is negative, so that $H(n,k)$ cannot be an integer.

Moreover, a study of the $p$-adic valuation of the harmonic numbers was initiated by Eswarathasan and Levine~\cite{MR1129989}.
They conjectured that for any prime number $p$ the set $\mathcal{J}_p$ of all positive integers $n$ such that $\nu_p(H_n) > 0$ is finite.
Although Boyd~\cite{MR1341721} gave a probabilistic model predicting that $\#\mathcal{J}_p = O(p^2(\log \log p)^{2 + \varepsilon})$, for any $\varepsilon > 0$, and Sanna~\cite{MR3486261} proved that $\mathcal{J}_p$ has asymptotic density zero, the conjecture is still open.
Another result of Sanna~\cite{MR3486261} is that $\nu_p(H_n) = -\lfloor \log_p n\rfloor$ for any $n$ in a subset $\mathcal{S}_p$ of the positive integers with logarithmic density greater than $0.273$.

In this paper, we study the $p$-adic valuation of $H(n,k)$.
Let $s(n,k)$ denotes an unsigned Stirling number of the first kind~\cite[\S6.1]{MR1397498}, i.e., $s(n,k)$ is the number of permutations of $\{1,\ldots,n\}$ with exactly $k$ disjoint cycles.
Then $H(n,k)$ and $s(n,k)$ are related by the following easy identity.

\begin{lem}\label{lem:Hnksnk}
For all integers $n \geq k \geq 1$, we have $H(n,k) = s(n+1,k+1)/n!$.
\end{lem}

In light of Lemma~\ref{lem:Hnksnk}, and since the $p$-adic valuation of the factorial is given by the formula~\cite[p.~517, 4.24]{MR1397498}
\begin{equation*}
\nu_p(n!) = \frac{n - s_p(n)}{p - 1} ,
\end{equation*}
where $s_p(n)$ is the sum of digits of the base $p$ representation of $n$, it follows that 
\begin{equation}\label{equ:Hnksn1k1}
\nu_p(H(n,k)) = \nu_p(s(n+1,k+1)) - \frac{n - s_p(n)}{p - 1} ,
\end{equation}
hence the study of $\nu_p(H(n,k))$ is equivalent to the study of $\nu_p(s(n + 1, k + 1))$.
That explains the title of this paper.

In this regard, $p$-adic valuations of sequences with combinatorial meanings have been studied by several authors (see, e.g., \cite{MR1667454, MR1337793, MR3190001, MR3275869, MR2333145, San16}).
In particular, the $p$-adic valuation of Stirling numbers of the second kind have been extensively studied~\cite{MR2410117, MR2192240, MR2926560, MR1285745, MR2721543}.
On the other hand, very few seems to be known about the $p$-adic valuation of Stirling numbers of the first kind. 
Indeed, up to our knowledge, the only systematic work on this topic is due to Lengyel~\cite{MR3283168}.
Among several results, he showed (see the proof of \cite[Theorem~1.2]{MR3283168}) that
\begin{equation}\label{eq:Lengyel_lower_bound}
\nu_p(H(n,k)) > -k \log_p n + O_k(1) ,
\end{equation}
for all prime numbers $p$ and all integers $n \geq k \geq 1$.

The main aim of this article is to provide an upper bound for $\nu_p(H(n,k))$.
In this respect, we believe that inequality (\ref{eq:Lengyel_lower_bound}) is nearly optimal, and our Theorem~\ref{thm:ubound} confirms this in the special case when the base $p$ representation of $n$ starts with the base $p$ representation of $k - 1$.
We also formulate the following:
\begin{conj}
For any prime number $p$ and any integer $k \geq 1$, there exists a constant $c = c(p,k) > 0$ such that $\nu_p(H(n,k)) < -c \log n$ for all sufficiently large integers $n$.
\end{conj}

\section{Notation and Main results}

Before state our results, we need to introduce some notation and definition. 
For any prime number $p$, we write
\begin{equation}\label{equ:theyaredigits}
\langle a_0, \ldots, a_v \rangle_p := \sum_{i = 0}^v a_i p^{v - i}, \text{ where } a_0, \ldots, a_v \in \{0, \ldots, p - 1\}, \; a_0 \neq 0 ,
\end{equation}
to denote a base $p$ representation.
In particular, hereafter, the restrictions of (\ref{equ:theyaredigits}) on $a_0, \ldots, a_v$ will be implicitly assumed any time we will write something like $\langle a_0, \ldots, a_v \rangle_p$.

We call \emph{$p$-tree of root $\langle a_0, \ldots, a_v \rangle_p$} a set of positive integers $\mathcal{T}$ such that:
\begin{enumerate}[label={\rm (\textsc{t}\arabic{*})}]
\item\label{item:t1} $\langle a_0, \ldots, a_v \rangle_p \in \mathcal{T}$;
\item\label{item:t2} If $\langle b_0, \ldots, b_u \rangle_p \in \mathcal{T}$ then $u \geq v$ and $b_i = a_i$ for $i = 0,\ldots,v$;
\item\label{item:t3} If $\langle b_0, \ldots, b_u \rangle_p \in \mathcal{T}$ and $u > v$ then $\langle b_0, \ldots, b_{u-1} \rangle_p \in \mathcal{T}$.
\end{enumerate}
Moreover, $\langle b_0, \ldots, b_u \rangle_p$ is a \emph{leaf} of $\mathcal{T}$ if $u > v$ and $\langle b_0, \ldots, b_{u-1} \rangle_p \in \mathcal{T}$ but $\langle b_0, \ldots, b_u \rangle_p \notin \mathcal{T}$.
The set of all leaves of $\mathcal{T}$ is denoted by $\mathcal{T}^\star$.
Finally, the \emph{girth} of $\mathcal{T}$ is the greatest integer $g$ such that for all $\langle b_0, \ldots, b_u \rangle_p \in \mathcal{T}$ we have $\langle b_0, \ldots, b_u, c \rangle_p \in \mathcal{T}$ for at least $g$ values of $c \in \{0, \ldots, p - 1\}$.

We are ready to state our results about the $p$-adic valuation of $H(n,k)$. 

\begin{thm}\label{thm:tree}
Let $p$ be a prime number and $k = \langle e_0, \ldots, e_t \rangle_p + 1 \ge 2$.
Then there exist a $p$-tree $\mathcal{T}_p(k)$ of root $k - 1$ and a nonnegative integer $W_p(k)$ such that for any integers $n = \langle e_0, \ldots, e_t, d_{t+1}, \ldots, d_s \rangle_p$ and $r \in [t + 1, s]$ we have:
\begin{enumerate}[label={\rm (\roman{*})}]
\item If $\langle e_0, \ldots, e_t, d_{t+1}, \ldots, d_r \rangle_p \in \mathcal{T}_p(k)$ then $\nu_p(H(n,k)) \geq W_p(k) + r - ks + 1$;
\item If $\langle e_0, \ldots, e_t, d_{t+1}, \ldots, d_r \rangle_p \in \mathcal{T}_p(k)^\star$ then $\nu_p(H(n,k)) = W_p(k) + r - ks$.
\end{enumerate}
Moreover, the girth of $\mathcal{T}_p(k)$ is less than $p^{0.835}$.
In particular, the girth of $\mathcal{T}_2(k)$ is equal~to~$1$.
\end{thm}

Note that the case $k=1$ has been excluded from the statement.
(As mentioned in the introduction, see~\cite{MR1341721, MR1129989, MR3486261} for results on the $p$-adic valuation of $H(n,1) = H_n$.)
Later, in Section~\ref{sec:computation}, we explain a method to effectively compute the elements of $\mathcal{T}_p(k)$ for given $p$ and $k$, and we also illustrate some examples of the results of these computations.

Lengyel~\cite[Theorem~2.5]{MR3283168} proved that for each integer $m \geq 2$ it holds
\begin{equation*}
\nu_2(s(2^m, 3)) = 2^m - 3m + 3
\end{equation*}
which, in light of identity (\ref{equ:Hnksn1k1}), is in turn equivalent to
\begin{equation}\label{equ:lengyelv2}
\nu_2(H(2^m - 1, 2)) = 4 - 2m .
\end{equation}
As an application of Theorem~\ref{thm:tree}, we give a corollary that generalizes (\ref{equ:lengyelv2}) and provides a quite precise description of $\nu_2(H(n,2))$.

\begin{cor}\label{cor:2adic}
There exists a sequence $f_0, f_1, \ldots \in \{0, 1\}$ such that for any integer $n = \langle d_0, \ldots, d_s \rangle_2 \geq 2$ we have:
\begin{enumerate}[label={\rm (\roman{*})}]
\item If $d_0 = f_0, \ldots, d_s = f_s$, then $\nu_2(H(n,2)) \geq 1 - s$;
\item If $d_0 = f_0, \ldots, d_{r-1} = f_{r - 1}$, and $d_r \neq f_r$, for some positive integer $r \leq s$, then $\nu_2(H(n,2)) = r - 2s$.
\end{enumerate}
Precisely, the sequence $f_0, f_1, \ldots$ can be computed recursively by $f_0 = 1$ and
\begin{equation}\label{equ:recf}
f_s = \begin{cases}
1 & \text{ if } \nu_2(H(\langle f_0, \ldots, f_{s-1}, 1 \rangle_2, 2)) \geq 1 - s, \\
0 & \text{ otherwise},
\end{cases}
\end{equation}
for any positive integer $s$.
In particular, $f_0 = 1$, $f_1 = 1$, $f_2 = 0$.
\end{cor}
Note that (\ref{equ:lengyelv2}) is indeed a consequence of Corollary~\ref{cor:2adic}.
In fact, on the one hand, for $m = 2$ the identity (\ref{equ:lengyelv2}) can be checked quickly.
On the other hand, for any integer $m \geq 3$ we have $2^m - 1 = \langle d_0, \ldots, d_{m-1} \rangle_2$ with $d_0 = \cdots = d_{m-1} = 1$, so that $d_0 = f_0$, $d_1 = f_1$, and $d_2 \neq f_2$, hence (\ref{equ:lengyelv2}) follows from Corollary~\ref{cor:2adic}(ii), with $s = m - 1$ and $r = 2$.

Finally, we obtain the following upper bound for $\nu_p(H(n,k))$.

\begin{thm}\label{thm:ubound}
Let $p$ be a prime number, $k = \langle e_0, \ldots, e_t \rangle_p + 1 \ge 2$, and $x \geq (k - 1)p$.
Then the inequality
\begin{equation*}
\nu_p(H(n,k)) < -(k - 1)(\log_p n - \log_p(k - 1) - 1)
\end{equation*}
holds for all $n = \langle e_0, \ldots, e_t, d_{t+1}, \ldots, d_s \rangle_p \in [(k-1)p, x]$, but at most $3x^{0.835}$ exceptions.
\end{thm}

\section{Preliminaries}

Let us start by proving the identity claimed in Lemma \ref{lem:Hnksnk}.

\begin{proof}[Proof of Lemma \ref{lem:Hnksnk}]
By~\cite[Eq.~6.11]{MR1397498} and $s(n+1,0) = 0$, we have the polynomial identity
\begin{equation*}
\prod_{i = 1}^n (X + i) = \sum_{k = 0}^{n} s(n + 1, k + 1) X^k ,
\end{equation*}
hence
\begin{equation*}
1 + \sum_{k = 1}^n H(n,k) X^k = \prod_{i = 1}^n \left(\frac{X}{i} + 1\right) = \frac1{n!}\prod_{i = 1}^n (X + i) = \sum_{k = 0}^{n} \frac{s(n + 1, k + 1)}{n!} X^k
\end{equation*}
and the claim follows.
\end{proof}

From here later, let us fix a prime number $p$ and let $k = \langle e_0, \ldots, e_t \rangle_p + 1 \geq 2$ and $n = \langle d_0, \ldots, d_s \rangle_p$ be positive integers with $s \geq t + 1$ and $d_i = e_i$ for $i = 0,\ldots,t$.
For any $a_0, \ldots, a_v \in \{0, \ldots, p - 1\}$, define
\begin{equation*}
B_p(a_0, \ldots, a_v) := \langle a_0, \ldots, a_v \rangle_p - \langle a_0, \ldots, a_{v-1} \rangle_p,
\end{equation*}
where by convention $\langle a_0, \ldots, a_{v-1} \rangle_p = 0$ if $v = 0$, and also
\begin{equation*}
\mathcal{B}_p(a_0, \ldots, a_v) := \big\{c_p(i) : i = 1, \ldots, B_p(a_0, \ldots, a_v)\big\}
\end{equation*}
where $c_p(1) < c_p(2) < \cdots$ denotes the sequence of all positive integers not divisible by $p$.
Lastly, put
\begin{equation*}
\mathcal{A}_p(n, v) := \big\{m \in \{1, \ldots, n\} : \nu_p(m) = s - v \big\} ,
\end{equation*}
for each integer $v \geq 0$.
The next lemma relates $\mathcal{A}_p(n, v)$ and $\mathcal{B}_p(d_0, \ldots, d_v)$.

\begin{lem}\label{lem:Apnv}
For each nonnegative integer $v \leq s$, we have
\begin{equation*}
\mathcal{A}_p(n, v) = \big\{j p^{s - v} : j \in \mathcal{B}_p(d_0, \ldots, d_v) \big\} .
\end{equation*}
In particular, $\#\mathcal{A}_p(n, v) = B_p(d_0, \ldots, d_v)$ and $\mathcal{A}_p(n, v)$ depends only on $p,s,d_0, \ldots, d_v$.
\end{lem}
\begin{proof}
For $m \in \{1, \ldots, n\}$, we have $m \in \mathcal{A}_p(n, v)$ if and only if $p^{s-v} \mid n$ but $p^{s-v+1} \nmid n$.
Therefore,
\begin{align*}
\# \mathcal{A}_p(n, v) &= \left\lfloor \frac{n}{p^{s-v}}\right\rfloor - \left\lfloor \frac{n}{p^{s-v+1}}\right\rfloor = \left\lfloor \sum_{i = 0}^s d_i p^{v - i} \right\rfloor - \left\lfloor \sum_{i = 0}^s d_i p^{v - i - 1} \right\rfloor \\
&= \sum_{i = 0}^v d_i p^{v - i} - \sum_{i = 0}^{v - 1} d_i p^{v - i - 1} = \langle d_0, \ldots, d_v \rangle_p - \langle d_0, \ldots, d_{v-1} \rangle_p \\
&= B_p(d_0, \ldots, d_v) ,
\end{align*}
and
\begin{align*}
\mathcal{A}_p(n, v) &= \big\{c_p(i) p^{s - v} : i = 1,\ldots,\#\mathcal{A}_p(n,v) \big\} \\
&= \big\{c_p(i) p^{s - v} : i = 1,\ldots,B_p(d_0,\ldots,d_v) \big\} \\
&= \big\{j p^{s - v} : j \in \mathcal{B}_p(d_0, \ldots, d_v) \big\} ,
\end{align*}
as claimed.
\end{proof}

Before stating the next lemma, we need to introduce some additional notation.
First, we define
\begin{equation*}
\mathcal{C}_p(n, k) := \bigcup_{v = 0}^t \mathcal{A}_p(n, v) \;\;\text{ and }\;\; \Pi_p(k) := \prod_{j \in \mathcal{C}_p(n, k)} \frac1{\free_p(j)} ,
\end{equation*}
where $\free_p(m) := m / p^{\nu_p(m)}$ for any positive integer $m$.
Note that, since $d_i = e_i$ for $i=0,\ldots,t$, from Lemma~\ref{lem:Apnv} it follows easily that indeed $\Pi_p(k)$ depends only on $p$ and $k$, and not on $n$. 
Then we put
\begin{equation*}
U_p(k) := \sum_{v = 0}^t B_p(e_0, \ldots, e_v) v + t + 1 ,
\end{equation*}
while, for $a_0, \ldots, a_{t + v + 1} \in \{0, \ldots, p - 1\}$, with $v \geq 0$ and $a_i = e_i$ for $i = 0, \ldots, t$, we set
\begin{equation*}
H_p^\prime(a_0, \ldots, a_{t+v}) := \sum_{\substack{0 \leq v_1, \ldots, v_k \leq t + v \\ v_1 + \cdots + v_k = U_p(k) + v}} \sum_{\substack{j_1 / p^{v_1} < \cdots < j_k / p^{v_k} \\ j_1 \in \mathcal{B}_p(a_0, \ldots, a_{v_1}), \ldots, j_k \in \mathcal{B}_p(a_0, \ldots, a_{v_k})}} \frac1{j_1 \cdots j_k}
\end{equation*}
and
\begin{equation*}
H_p(a_0, \ldots, a_{t+v+1}) := H_p^\prime(a_0, \ldots, a_{t+v}) + \Pi_p(k) \sum_{j \in \mathcal{B}_p(a_0, \ldots, a_{t+v+1})} \frac1{j} .
\end{equation*}
Note that $\nu_p(H_p(a_0, \ldots, a_{t+v+1})) \geq 0$, this fact will be fundamental later.

The following lemma gives a kind of $p$-adic expansion for $H(n,k)$.
We use $O(p^v)$ to denote a rational number with $p$-adic valuation greater than or equal to $v$.

\begin{lem}\label{lem:vpHnk}
We have
\begin{equation*}
H(n, k) = \sum_{v = 0}^{s - t - 1} H_p(d_0, \ldots, d_{t + v + 1})\cdot p^{v - ks + U_p(k)} + O\big(p^{s - t - ks + U_p(k)}\big) .
\end{equation*}
\end{lem}
\begin{proof}
Clearly, we can write
\begin{equation*}
H(n, k) = \sum_{v = 0}^{V_p(n, k)} J_p(n, k, v) \cdot p^{v - V_p(n, k)} ,
\end{equation*}
where
\begin{equation*}
V_p(n, k) := \max\!\big\{\nu_p(i_1 \cdots i_k) : 1 \leq i_1 < \cdots < i_k \leq n \big\} ,
\end{equation*}
and
\begin{equation*}
J_p(n, k, v) := \sum_{\substack{1 \leq i_1 < \cdots < i_k \leq n \\ \nu_p(i_1 \cdots i_k) = V_p(n, k) - v}} \frac1{\free_p(i_1 \cdots i_k)} ,
\end{equation*}
for each nonnegative integer $v \leq V_p(n,k)$.

We shall prove that $V_p(n, k) = ks - U_p(k)$.
On the one hand, we have
\begin{align}\label{equ:Bpkminus1}
\sum_{v = 0}^t B_p(e_0, \ldots, e_v) &= \sum_{v = 0}^t\big(\langle e_0, \ldots, e_v\rangle_p - \langle e_0, \ldots, e_{v-1}\rangle_p\big) \\
&= \langle e_0, \ldots, e_t\rangle_p = k - 1 . \nonumber
\end{align}
On the other hand, by (\ref{equ:Bpkminus1}) and thanks to Lemma~\ref{lem:Apnv}, we obtain
\begin{align}\label{equ:Cpnk}
\# \mathcal{C}_p(n,k) &= \sum_{v = 0}^t \# \mathcal{A}_p(n, v) = \sum_{v = 0}^t B_p(e_0, \ldots, e_v) = k - 1.
\end{align}
Hence, in order to maximize $\nu_p(i_1 \cdots i_k)$ for positive integers $i_1 < \cdots < i_k \leq n$, we have to choose $i_1, \ldots, i_k$ by picking all the $k - 1$ elements of $\mathcal{C}_p(n, k)$ and exactly one element from $\mathcal{A}_p(n, t + 1)$.
Therefore, using again (\ref{equ:Bpkminus1}) and Lemma~\ref{lem:Apnv}, we get
\begin{align}\label{equ:Vpnk}
V_p(n, k) &= \sum_{v = 0}^t \#\mathcal{A}_p(n, v) (s - v) + (s - t - 1) \\
&= \sum_{v = 0}^t B_p(e_0, \ldots, e_v) (s - v) + (s - t - 1) \nonumber\\
&= \left(\sum_{v = 0}^t B_p(e_0, \ldots, e_v) + 1\right) \! s - U_p(k) \nonumber\\
&= ks - U_p(k) , \nonumber
\end{align}
as desired.

Similarly, if $\nu_p(i_1 \cdots i_k) = V_p(n,k) - v$, for some positive integers $i_1 < \cdots < i_k \leq n$ and some nonnegative integer $v \leq s - t - 1$, then only two cases are possible: $\nu_p(i_1), \ldots, \nu_p(i_k) \geq s - t - v$; or $i_1, \ldots, i_k$ consist of all the $k - 1$ elements of $\mathcal{C}_p(n, k)$ and one element of $\mathcal{A}_p(n, t + v + 1)$.
As~a~consequence,
\begin{equation}\label{equ:Jpnkv}
J_p(n, k, v) = \sum_{\substack{1 \leq i_1 < \cdots < i_k \leq n \\ \nu_p(i_1 \cdots i_k) = V_p(n, k) - v \\ \nu_p(i_1), \ldots, \nu_p(i_k) \geq s - t - v}} \frac1{\free_p(i_1 \cdots i_k)} + \Pi_p(k) \sum_{i \in \mathcal{A}_p(n, t + v + 1)} \frac1{\free_p(i)} ,
\end{equation}
for all nonnegative integers $v \leq s - t - 1$.

By putting $v_\ell := s - \nu_p(i_\ell)$ and $j_\ell := \free_p(i_\ell)$ for $\ell = 1,\ldots,k$, the first sum of (\ref{equ:Jpnkv}) can be rewritten as
\begin{align*}
&\phantom{mmm}\sum_{\substack{0 \leq v_1, \ldots, v_k \leq t + v \\ (s - v_1) + \cdots + (s - v_k) = V_p(n,k) - v}} \sum_{\substack{i_1 < \cdots < i_k \\ i_1 \in \mathcal{A}_p(n, v_1), \ldots, i_k \in \mathcal{A}_p(n, v_k)}} \frac1{\free_p(i_1 \cdots i_k)} \\
&= \sum_{\substack{0 \leq v_1, \ldots, v_k \leq t + v \\ v_1 + \cdots + v_k = U_p(k) + v}} \sum_{\substack{j_1 / p^{v_1} < \cdots < j_k / p^{v_k} \\ j_1 \in \mathcal{B}_p(d_0, \ldots, d_{v_1}), \ldots, j_k \in \mathcal{B}_p(d_0, \ldots, d_{v_k})}} \frac1{j_1 \cdots j_k} = H^\prime_p(d_0, \ldots, d_{t+v}) ,
\end{align*}
where we have also made use of (\ref{equ:Vpnk}) and Lemma~\ref{lem:Apnv}, hence
\begin{equation}\label{equ:JpnkvHp}
J_p(n, k, v) = H_p(d_0, \ldots, d_{t + v + 1}) ,
\end{equation}
for any nonnegative integer $v \leq s - t - 1$.

At this point, being $s > t$, by (\ref{equ:Vpnk}) it follows that $V_p(n, k) > s - t - 1$, hence
\begin{equation}\label{equ:Hnkvst1}
H(n, k) = \sum_{v = 0}^{s - t - 1} J_p(n, k, v) \cdot p^{v - ks + U_p(k)} + O\big(p^{s - t - ks + U_p(k)}\big) ,
\end{equation}
since clearly $\nu_p(J_p(n, k, v)) \geq 0$ for any nonnegative integer $v \leq V_p(n, k)$.

In conclusion, the claim follows from (\ref{equ:JpnkvHp}) and (\ref{equ:Hnkvst1}).
\end{proof}

Finally, we need two lemmas about the number of solutions of some congruences.
For rational numbers $a$ and $b$, we write $a \equiv b \bmod p$ to mean that $\nu_p(a - b) > 0$.

\begin{lem}\label{lem:harm}
Let $r$ be a rational number and let $x, y$ be positive integers with $y < p$.
Then the number of integers $v \in [x, x+y]$ such that $H_v \equiv r \bmod p$ is less than $\frac{3}{2}y^{2/3} + 1$.
\end{lem}
\begin{proof}
The case $r = 0$ is proved in \cite[Lemma~2.2]{MR3486261} and the proof works exactly in the same way even for $r \neq 0$.
\end{proof}

\begin{lem}\label{lem:cpicong}
Let $q$ be a rational number and let $a$ be a positive integer.
Then the number of $d \in \{0, \ldots, p-1\}$ such that
\begin{equation}\label{equ:cpqp}
\sum_{i = a}^{a + d} \frac1{c_p(i)} \equiv q \bmod p
\end{equation}
is less than $p^{0.835}$.
\end{lem}
\begin{proof}
It is easy to see that there exists some $h \in \{0,\ldots,p-2\}$ such that
\begin{equation*}
c_p(i) = \begin{cases}
c_p(a) + i - a & \text{ for } i = a,\ldots,a+h, \\
c_p(a) + i - a + 1 & \text{ for } i = a+h+1,\ldots,a+p-1 .
\end{cases}
\end{equation*}
Therefore, by putting $x := c_p(a)$, $y := h$, and $r := q + H_{x - 1}$ in Lemma~\ref{lem:harm}, we get that the number of $d \leq h$ satisfying (\ref{equ:cpqp}) is less than $\frac{3}{2}h^{2/3} + 1$.
Similarly, by putting $x := c_p(a) + h + 2$, $y := p - h - 2$, and
\begin{equation*}
r := q + H_{x - 1} - \sum_{i = a}^{a + h} \frac1{c_p(i)}
\end{equation*}
in Lemma~\ref{lem:harm}, we get that the number of $d \in [h + 1, p - 1]$ satisfying (\ref{equ:cpqp}) is less than $\frac{3}{2}(p - h - 2)^{2/3} + 1$.
Thus, letting $N$ be the number of $d \in \{0, \ldots, p - 1\}$ that satisfy (\ref{equ:cpqp}), we have
\begin{equation*}
N \leq \frac{3}{2}h^{2/3}+1+\frac{3}{2}(p - h - 2)^{2/3} + 1 \leq 3\!\left(\frac{p - 2}{2}\right)^{2/3} + 2 .
\end{equation*}
Furthermore, it is clear the $d$ and $d+1$ cannot both satisfy (\ref{equ:cpqp}), hence $N \leq \left\lceil p / 2 \right\rceil$.
Finally, a little computation shows that the maximum of
\begin{equation*}
\log_p\!\left(\min\!\left(3\!\left(\frac{p - 2}{2}\right)^{2/3} + 2, \left\lceil \frac{p}{2} \right\rceil\right)\right)
\end{equation*}
is obtained for $p = 59$ and is less than $0.835$, hence the claim follows.
\end{proof}

\section{Proof of Theorem~\ref{thm:tree}}

Now we are ready to prove Theorem~\ref{thm:tree}.
For any $a_0, \ldots, a_{t+u+1} \in \{0,\ldots,p-1\}$, with $u \geq 0$ and $a_i = e_i$ for $i=0,\ldots,t$, let
\begin{equation*}
\Sigma_p(a_0, \ldots, a_{t+u+1}) := \sum_{v = 0}^u H_p(a_0, \ldots, a_{t+v+1})\cdot p^v .
\end{equation*}
Furthermore, define the sequence of sets $\mathcal{T}_p^{(0)}(k), \mathcal{T}_p^{(1)}(k), \ldots$ as follows: $\mathcal{T}_p^{(0)}(k) := \left\{\langle e_0, \ldots, e_t \rangle_p\right\}$, and for any integer $u \geq 0$ put $\langle a_0, \ldots, a_{t+u+1} \rangle_p \in \mathcal{T}_p^{(u + 1)}(k)$ if and only if $\langle a_0, \ldots, a_{t+u} \rangle_p \in \mathcal{T}_p^{(u)}(k)$ and $\nu_p(\Sigma_p(a_0, \ldots, a_{t+u+1})) \geq u + 1$.
At this point, setting 
\begin{equation*}
\mathcal{T}_p(k) := \bigcup_{u=0}^\infty \mathcal{T}_p^{(u)}(k) ,
\end{equation*}
it is straightforward to see that $\mathcal{T}_p(k)$ is a $p$-tree of root $\langle e_0, \ldots, e_t\rangle_p$.
Put $W_p(k) := U_p(k) - t - 1$.

If $\langle d_0, \ldots, d_r \rangle_p \in \mathcal{T}_p(k)$ then, by the definition of $\mathcal{T}_p(k)$, we have $\nu_p(\Sigma_p(d_0, \ldots, d_r)) \geq r-t$.
Therefore, by Lemma~\ref{lem:vpHnk} it follows that $\nu_p(H(n,k)) \geq W_p(k) + r - ks + 1$, and this proves (i).

If $\langle d_0, \ldots, d_r \rangle_p \in \mathcal{T}_p(k)^\star$ then $r > t$ and $\langle d_0, \ldots, d_{r-1}\rangle \in \mathcal{T}_p(k)$ but $\langle d_0, \ldots, d_r \rangle \notin \mathcal{T}_p(k)$, so that
\begin{equation}\label{equ:vpSpd0}
\nu_p(\Sigma_p(d_0, \ldots, d_r)) \leq r - t - 1 .
\end{equation}
Now we distinguish between two cases.
If $r = t + 1$, then $\nu_p(\Sigma_p(d_0, \ldots, d_{t+1})) = 0$ and by Lemma~\ref{lem:vpHnk} we obtain $\nu_p(H(n,k)) = W_p(k) + r - ks$.
If $r > t + 1$ then by $\langle d_0, \ldots, d_{r-1}\rangle \in \mathcal{T}_p(k)$ we get that $\nu_p(\Sigma_p(d_0, \ldots, d_{r-1})) \geq r - t - 1$, which together with (\ref{equ:vpSpd0}) and
\begin{equation*}
\Sigma_p(d_0, \ldots, d_r) = \Sigma_p(d_0, \ldots, d_{r-1}) + H_p(d_0, \ldots, d_r)\cdot p^{r - t - 1}
\end{equation*}
implies that $\nu_p(\Sigma_p(d_0, \ldots, d_r)) = r - t - 1$, hence by Lemma~\ref{lem:vpHnk} we get $\nu_p(H(n,k)) = W_p(k) + r - ks$, and also (ii) is proved.

It remains only to bound the girth of $\mathcal{T}_p(k)$.
Let $u$ be a nonnegative integer and pick $\langle a_0, \ldots, a_{t + u} \rangle_p \in \mathcal{T}_p^{(u)}(k)$.
By the definition of $\mathcal{T}_p^{(u+1)}(k)$, we have $\langle a_0, \ldots, a_{t + u + 1} \rangle_p \in \mathcal{T}_p^{(u+1)}(k)$ if and only if $\nu_p(\Sigma_p(a_0, \ldots, a_{t+u+1})) \geq u + 1$, which in turn is equivalent to
\begin{align}\label{equ:longcong1}
&\sum_{v = 0}^{u-1} H_p(a_0, \ldots, a_{t+v+1})\cdot p^{v-u} + H_p^\prime(a_0, \ldots, a_{t+u}) + \Pi_p(k) \sum_{j \in \mathcal{B}_p(a_0, \ldots, a_{t+u+1})} \frac1{j} \\
&\equiv \sum_{v = 0}^u H_p(a_0, \ldots, a_{t+v+1})\cdot p^{v-u} \equiv 0 \bmod p \nonumber .
\end{align}
Using the definition of $\mathcal{B}_p(a_0, \ldots, a_{t+u+1})$ and the facts that
\begin{equation*}
B_p(a_0, \ldots, a_{t+u+1}) = a_{t + u + 1} + (p-1)\sum_{v = 0}^u a_v p^{u - v} ,
\end{equation*}
and $\nu_p(\Pi_p(k)) = 0$, we get that (\ref{equ:longcong1}) is equivalent to
\begin{align}\label{equ:longcong2}
&\phantom{mmm}\sum_{i = a}^{a + a_{t+u+1}} \frac1{c_p(i)} \equiv -\sum_{i = 1}^{a - 1} \frac1{c_p(i)} \nonumber\\
&-\frac1{\Pi_p(k)}\!\left(\sum_{v = 0}^{u-1} H_p(a_0, \ldots, a_{t+v+1})\cdot p^{v-u} + H_p^\prime(a_0, \ldots, a_{t+u}) \right) \bmod p ,
\end{align}
where
\begin{equation*}
a := (p-1)\sum_{v = 0}^u a_v p^{u - v} .
\end{equation*}
Note that both $a$ and the right-hand side of (\ref{equ:longcong2}) do not depend on $a_{t + u + 1}$.
As a consequence, by Lemma~\ref{lem:cpicong} we get that $\langle a_0, \ldots, a_{t + u + 1} \rangle_p \in \mathcal{T}_p^{(u+1)}(k)$ for less than $p^{0.835}$ values of $a_{t + u + 1} \in \{0, \ldots, p - 1\}$.
Thus the girth of $\mathcal{T}_p(k)$ is less than $p^{0.835}$.

Finally, consider the case $p = 2$.
Obviously, $1 / c_2(i) \equiv 1 \bmod 2$ for any positive integer $i$, while the right-hand side of (\ref{equ:longcong2}) is equal to $0$ or $1$ (mod $2$).
Therefore, there exists one and only one choice of $a_{t+u+1} \in \{0, 1\}$ such that (\ref{equ:longcong2}) is satisfied.
This means that the girth of $\mathcal{T}_2(k)$ is exactly $1$.

The proof is complete.

\section{The computation of $\mathcal{T}_p(k)$}\label{sec:computation}

Given $p$ and $k$, it might be interesting to effectively compute the elements of $\mathcal{T}_p(k)$.
Clearly, $\mathcal{T}_p(k)$ could be infinite --- and this is indeed the case when $p = 2$, since by Theorem~\ref{thm:tree} we know that $\mathcal{T}_2(k)$ has girth $1$ --- hence the computation should proceed by first enumerating all the elements of $\mathcal{T}_p^{(0)}(k)$, then all the elements of $\mathcal{T}_p^{(1)}(k)$, and so on.
An obvious way to do this is using the recursive definition of the $\mathcal{T}_p^{(u)}(k)$'s.
However, it is easy to see how this method is quite complicated and impractical.
A better idea is noting that, taking $r = s$ in Theorem~\ref{thm:tree}, we have
\begin{align}\label{equ:Tprec}
\mathcal{T}_p^{(u+1)}(k) = \big\{\langle a_0, \ldots, a_{t+u}, b \rangle_p :\,& \langle a_0, \ldots, a_{t+u} \rangle_p \in \mathcal{T}_p^{(u)}(k),\\
&\nu_p(H(\langle a_0, \ldots, a_{t + u}, b\rangle_p, k)) \geq W_p(k) - (k-1)s + 1 \big\} \nonumber,
\end{align}
for all integers $u \geq 0$.
Therefore, starting from $\mathcal{T}_p^{(0)}(k) = \{\langle e_0, \ldots, e_t \rangle_p\}$, formula (\ref{equ:Tprec}) gives a way to compute recursively all the elements of $\mathcal{T}_p(k)$.
In particular, if $\mathcal{T}_p(k)$ is finite, then after sufficient computation one will get $\mathcal{T}_p^{(u)}(k) = \varnothing$ for some positive integer $u$, so the method actually proves that $\mathcal{T}_p(k)$ is finite.

The authors implemented this algorithm in \textsc{SageMath}, since it allows computations with arbitrary-precision $p$-adic numbers.
In particular, they found that $\mathcal{T}_3(2),\ldots,\mathcal{T}_3(6)$ are all finite sets, with respectively $8, 24, 16, 7, 23$ elements, while the cardinality of $\mathcal{T}_3(7)$ is at least $43$.
Through these numerical experiments, it seems that, in general, $\mathcal{T}_p(k)$ does not exhibit any trivial structure (see Figures~\ref{fig:1}, \ref{fig:2}, \ref{fig:3}), hence the question of the finiteness of $\mathcal{T}_p(k)$ is probably a difficult one.

\section{Proof of Corollary~\ref{cor:2adic}}

Only for this section, let us focus on the case $p = 2$ and $k = 2$, so that $t = 0$, $e_0 = 1$, and $W_2(2) = 0$.
Thanks to Theorem~\ref{thm:tree} we know that $\mathcal{T}_2(2)$ has girth $1$.
Hence, it follows easily that there exists a sequence $f_0, f_1, \ldots \in \{0,1\}$ such that $\mathcal{T}^{(u)}_2(2) = \{\langle f_0, \ldots, f_u\rangle_2\}$ for all integers $u \geq 0$.
In particular, $f_0 = e_0 = 1$.
At this point, (i) and (ii) are direct consequences of Theorem~\ref{thm:tree}, while the recursive formula (\ref{equ:recf}) is just a special case of~(\ref{equ:Tprec}).

\section{Proof of Theorem~\ref{thm:ubound}}

It is easy to see that $\#\mathcal{T}_p^{(u)}(k) < p^{0.835 u}$, for any positive integer $u$.

On the one hand, the number of $n = \langle e_0, \ldots, e_t, d_{t+1}, \ldots, d_s \rangle_p \in [(k-1)p, x]$ such that $\langle d_0, \ldots, d_r \rangle_p \notin \mathcal{T}_p(k)^\star$ for any integer $r \in [t+1, s]$ is less than
\begin{equation*}
\sum_{u = 1}^{\lfloor \log_p x \rfloor - t} \#\mathcal{T}_p^{(u)}(k) < \sum_{u = 1}^{\lfloor \log_p x \rfloor} p^{0.835 u} < 3 x^{0.835} .
\end{equation*}

On the other hand, if $n = \langle e_0, \ldots, e_t, d_{t+1}, \ldots, d_s \rangle_p \in [(k-1)p, x]$ and $\langle d_0, \ldots, d_r \rangle_p \in \mathcal{T}_p(k)^\star$ for some integer $r \in [t + 1, s]$, then by Theorem~\ref{thm:tree}(ii) we get that
\begin{align*}
\nu_p(H(n,k)) &= W_p(k) + r - ks \leq W_p(k) - (k - 1)s = \sum_{v = 0}^t B_p(e_0, \ldots, e_v) v - (k - 1)s \\
&\leq \sum_{v = 0}^t B_p(e_0, \ldots, e_v) t - (k - 1)s < -(k - 1)(\log_p n - \log_p (k - 1) - 1) ,
\end{align*}
where we have made use of (\ref{equ:Bpkminus1}) and the inequalities $s > \log_p n - 1$ and $t \leq \log_p (k - 1)$.

\section{Figures}

\begin{figure}[h]
\centering
\includegraphics[scale=1.00]{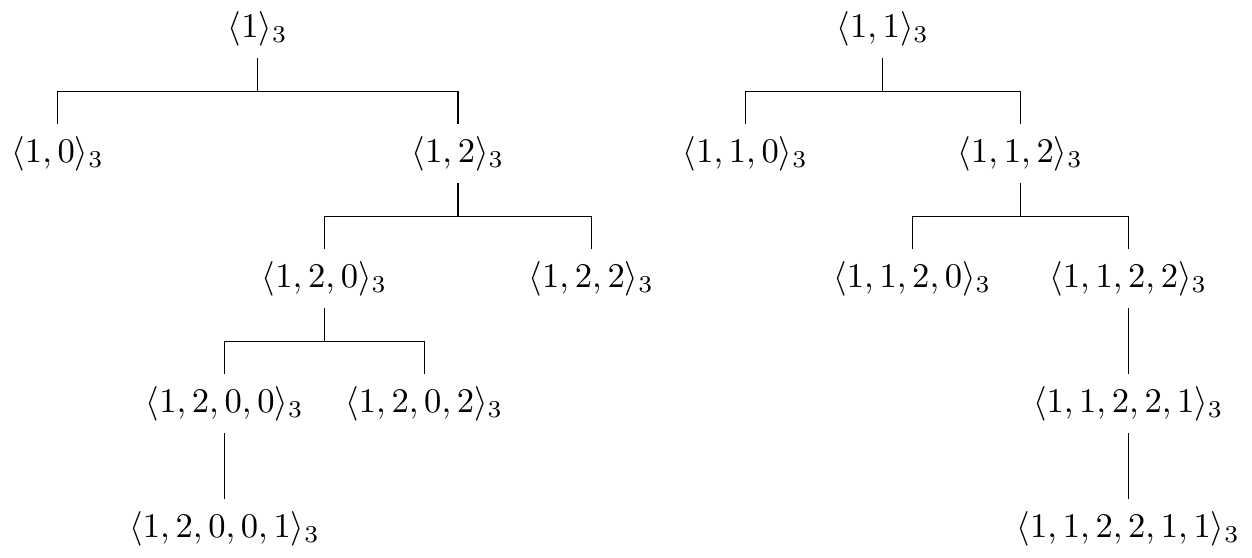}
\caption{The $8$ elements of $\mathcal{T}_3(2)$ (left tree), and the $7$ elements of~$\mathcal{T}_3(5)$~(right tree).}
\label{fig:1}
\end{figure}

\begin{figure}[h]
\centering
\includegraphics[scale=0.75]{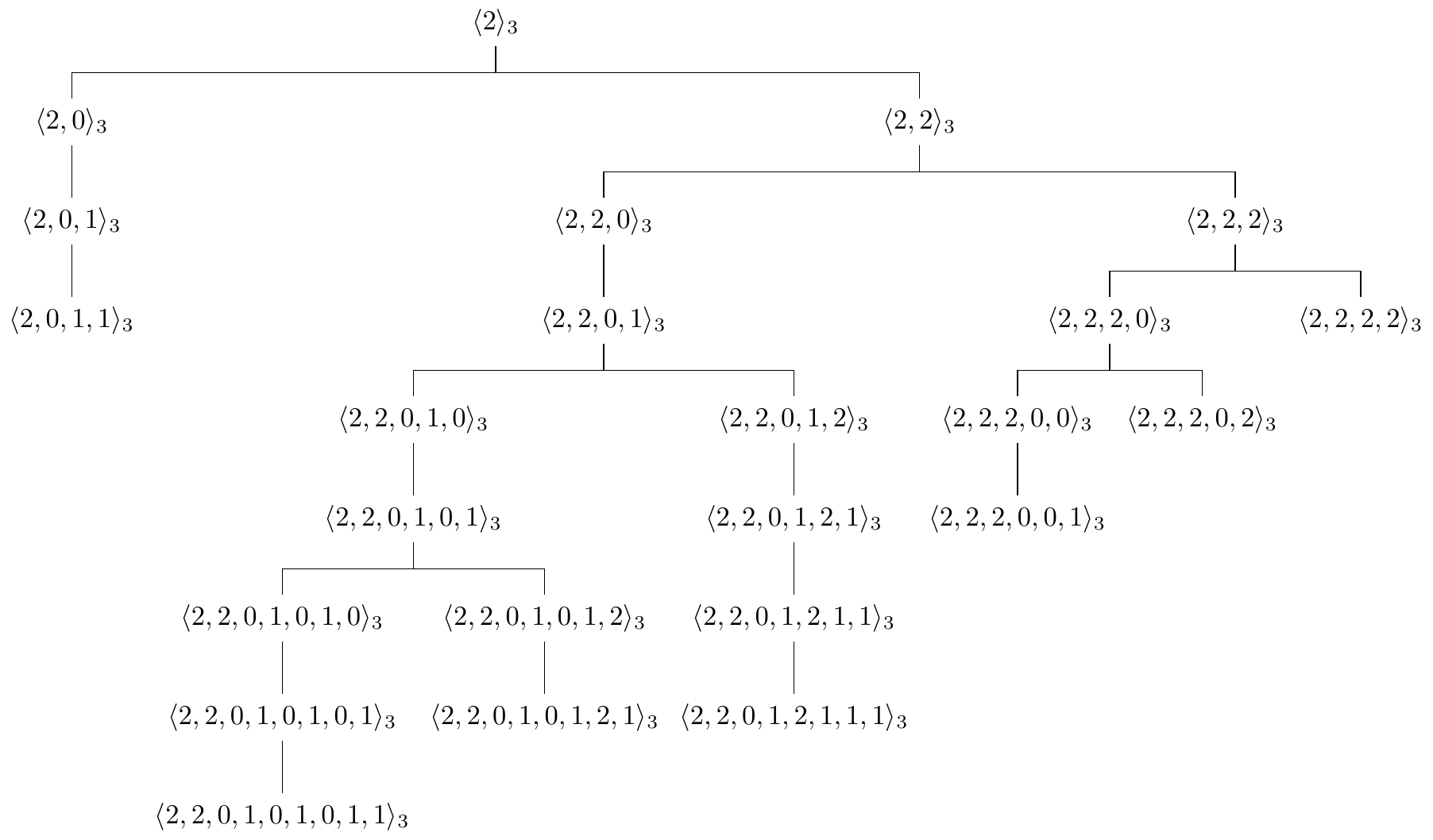}
\caption{The $24$ elements of $\mathcal{T}_3(3)$.}
\label{fig:2}
\end{figure}

\begin{figure}[h]
\centering
\includegraphics[scale=0.75]{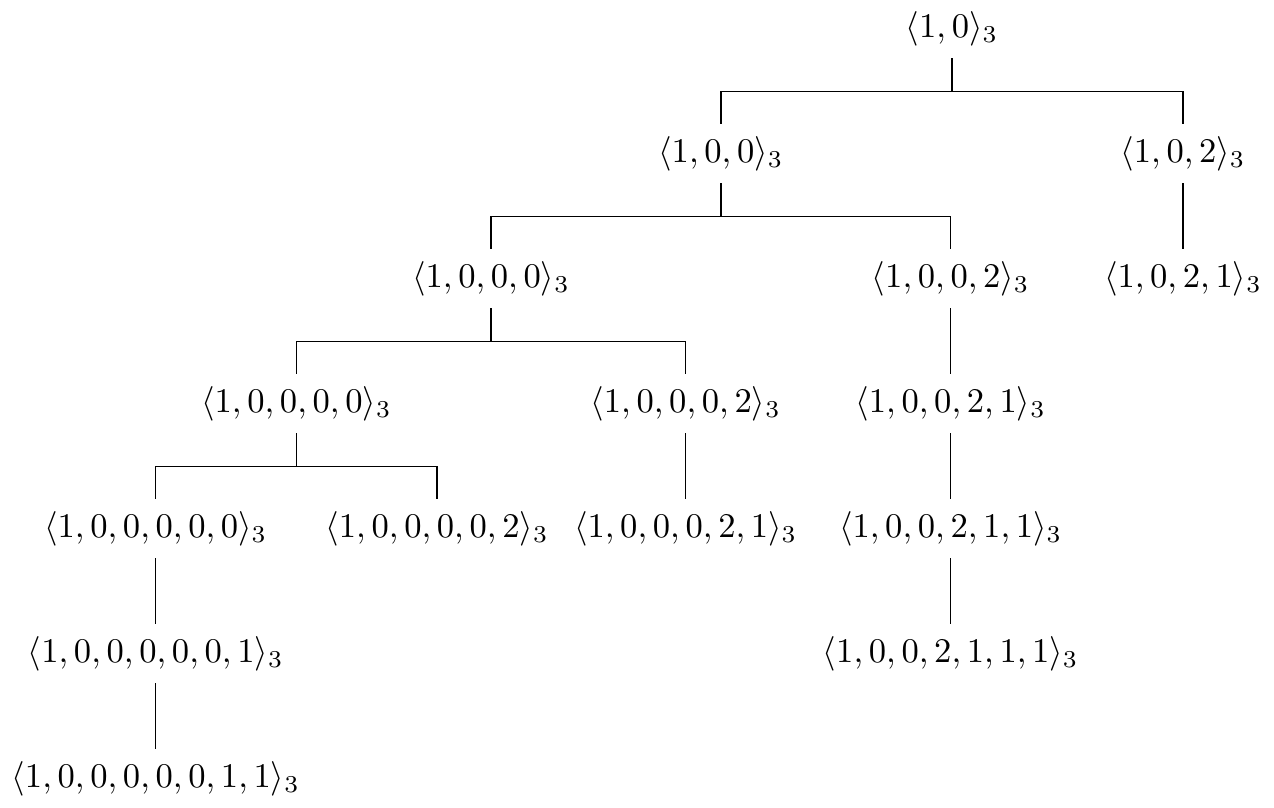}
\caption{The $16$ elements of $\mathcal{T}_3(4)$.}
\label{fig:3}
\end{figure}

\providecommand{\bysame}{\leavevmode\hbox to3em{\hrulefill}\thinspace}

\end{document}